\documentclass[reqno]{amsart}

\usepackage{amsmath}
\usepackage{amsfonts}
\usepackage{amssymb}
\usepackage{amsthm}
\usepackage{mathenv}
\usepackage{enumerate}
\usepackage{graphicx}

\newcommand{\oemph}[1]{\textit{#1}}
\newcommand{\norm}[2]{\left\| {#1}\right\| _{#2}}
\newcommand{\nnorm}[2]{\| {#1} \| _{#2}}

\newcommand{\all}[2]{\left \{\,{#1}\,:\,{#2}\,\right \}}

\newcommand{\Eref}[1]{\eqref{#1}}
\newcommand{\eref}[1]{\eqref{#1}}
\newcommand{\Sref}[1]{Section \ref{#1}}

\newcommand{\Fref}[1]{Figure \ref{#1}}
\newcommand{\fl}{}
\theoremstyle{plain} 
\newtheorem{theorem}{Theorem}[section]
\newtheorem{prop}[theorem]{Proposition}
\newtheorem{lem}[theorem]{Lemma} 
\newtheorem{cor}[theorem]{Corollary}

\theoremstyle{definition} 
\newtheorem{defn}[theorem]{Definition}
\newtheorem{rem}[theorem]{Remark}

\begin{document}

\title[Lyapunov exponents and exponential decay of correlations]{On the relation between 
Lyapunov exponents and exponential decay of correlations}

\author{Julia Slipantschuk, Oscar F.~Bandtlow and Wolfram Just}

\address{School of Mathematical Sciences, Queen Mary University of London, Mile End Road,
London E1 4NS, UK}

\email{j.slipantschuk@qmul.ac.uk,\, o.bandtlow@qmul.ac.uk, \,w.just@qmul.ac.uk}

\begin{abstract}
Chaotic dynamics with sensitive dependence on initial conditions may result in
exponential decay of correlation functions. We show that for one-dimensional interval 
maps the corresponding quantities, that is, Lyapunov exponents and exponential decay rates,
are related. For piecewise linear expanding Markov maps observed via piecewise analytic functions 
we provide explicit bounds of the decay rate in terms of the Lyapunov exponent.
In addition, we comment on similar relations for general piecewise smooth expanding maps.
\end{abstract}



\maketitle

\section{Introduction}\label{sec:1}

Chaotic motion accounts on the one hand for the well-known phenomenon of sensitive dependence on initial conditions, that is, exponentially
fast divergence of nearby orbits, and on the other hand for
the phenomenon of decay of correlations or mixing.
Both properties are intimately related with the observation
that even low-dimensional chaotic systems share common features
with random processes.

This intuitive picture
has been used as a basis to address some of the fundamental
questions arising in nonequilibrium statistical mechanics 
\cite{Kryl:79,Kubo_Sci86,FaIsVu_PLA90}.
In fact,
it is a simple exercise to show that topological mixing implies 
sensitive dependence on initial conditions. At the  measure-theoretical
level, however, relating Lyapunov exponents, the quantitative measures for sensitive dependence
on initial conditions, 
to decay rates of correlation functions is a more involved task. 
For instance, it is easy to construct 
simple maps with finite Lyapunov exponent and arbitrarily small correlation decay 
(see, for example, a Markov model described in \cite{Just_PLA90}). 

Thus, at a quantitative level it is tempting to explore in some detail in which way 
the rate of correlation decay is linked with Lyapunov exponents, as both 
quantities are supposed to have a 
common origin. In more general terms and from a wider perspective
this topic can be viewed as belonging to the realm of fluctuation dissipation relations for
nonequilibrium dynamics, where one cause, the underlying detailed dynamical structure, is
responsible for the approach to the stationary state, that is, for the decay of
correlations, but, at the same time, is responsible for fluctuation properties
at a microscopic level, or in our context for the sensitive dependence of initial conditions and 
a positive Lyapunov exponent. Furthermore, any relation 
between decay of correlations and Lyapunov exponents is of great practical interest, 
as the measurement of Lyapunov exponents, unlike the
rate of correlation decay, is notoriously difficult to determine in real world experiments \cite{Abar:96,KaSc:97}. 
In \cite{Spro:96} it was even suggested to take correlation decay rates as 
a meaningful approximation for Lyapunov exponents.

The problem we want to address can be illustrated by a very basic textbook example,
probably considered for the first time more than two decades ago 
\cite{BaHeMePo_PRA88}.
Consider a linear full branch map (see \Fref{fig:1}) on the unit interval $I=[0,1]$, that is,
a map $f:I\rightarrow I$ having a finite partition of $I$ into closed intervals $I_k$ with pairwise 
disjoint interior such that 
(i) for each $k$, we have $f(I_k)=I$, and
(ii) $f$ has constant slope $\gamma_k$ on each $I_k$.
The physical invariant measure is given by the Lebesgue measure and the 
Lyapunov exponent 
with respect to this measure
can be expressed in terms of the slopes
\begin{equation}\label{aa}
\Lambda = \sum_k |I_k| \ln|\gamma_k|,
\end{equation}
with $|I_k|=1/|\gamma_k|$ denoting the size of the interval $I_k$. 

The exponential rate of decay for correlation 
functions 
is determined by the negative logarithm of the second largest eigenvalue in modulus of the associated Perron-Frobenius operator.
In this setting it is well known (see, for example, \cite{MoSoOs_PTP81})
that eigenfunctions of the 
Perron-Frobenius operator are given by polynomials and
that the corresponding eigenvalues $\nu_m$ can be expressed as
\begin{equation*}\label{ab}
\nu_m = \sum_k \frac{1}{|\gamma_k|} \frac{1}{\gamma_k^{m}} = 
\sum_k |I_k| \frac{1}{\gamma_k^{m}} \,\qquad (m \geq 0),
\end{equation*}
with largest eigenvalue $\lambda_0=\nu_0=1$ and the subleading eigenvalue $\lambda_1$ 
being the second largest in modulus, $|\lambda_1| = \max\{|\nu_1|,\nu_2\}$. 
Thus, correlation functions of sufficiently smooth observables decay typically at an
exponential rate $\alpha=-\ln|\lambda_1|$.
Since $|\lambda_1|\geq \nu_2>0$ we obtain an upper bound for the decay rate 
$\alpha \leq -\ln \nu_2$ which can now be related to the Lyapunov exponent
\eqref{aa}. If we apply Jensen's inequality to the convex function 
$\varphi(x)=-\ln (x)$ we end up with
\begin{equation}\label{ac}
\alpha \leq -\ln \sum_k |I_k| \frac{1}{\gamma_k^{2}} \leq 
\sum_k |I_k| (-\ln \frac{1}{\gamma_k^{2}}) = 2 \Lambda \, .
\end{equation}
The estimate of the decay rate in \Eref{ac} has been based on $\nu_2$ which contains
positive terms only, even if the slopes have different signs. As a result the upper
bound is given by twice the Lyapunov exponent. If all slopes have the same sign,
say $\gamma_k>1$, then $\nu_1$ determines the subleading eigenvalue and the Lyapunov
exponent itself yields an upper bound for the decay rate, that is, $\alpha \leq \Lambda$. 

\begin{figure}[h!]
  \centering
    \includegraphics[width=.30\textwidth]{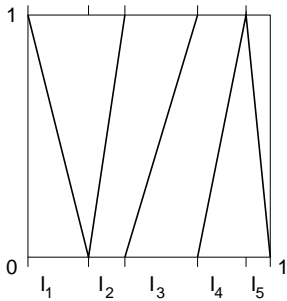}
  \caption{Diagrammatic view of a linear full branch map. \label{fig:1}}
\end{figure}

In this article we will address the question 
to which extent this simple reasoning can be generalised to a larger class of systems. 
From a mathematical perspective there is a considerable amount of 
literature on the existence of invariant measures of one-dimensional maps, 
the corresponding Lyapunov exponent, estimates for the spectra of the 
associated 
Perron-Frobenius  operator, and a possible relation between
Lyapunov exponents and the decay of correlation functions \cite{CoEc_JSP04}. Expansiveness of the
underlying map $f:I\to I$
is to some extent the key ingredient to establish a 
physical measure $\mu$, that is, an invariant measure which asymptotically characterises a large set of orbits, or in formal terms an invariant ergodic probability measure which is absolutely continuous with respect to Lebesque measure (see, for example, 
\cite{Live_AM95,Kell_IJBC99,Bala:00,BrLuSt_ASENS03}).

The Lyapunov exponent for a Lebesque typical point $x\in I$ is given by
\begin{equation}\label{eq:Lyapunov}
\Lambda = \int_I \ln|f'(x)| d\mu
= \lim_{n\to \infty} \frac{1}{n} \ln{|(f^n)^{\prime}(x)|.}  
\end{equation}
Regarding decay rates for correlation functions, a substantial part of the problem 
consists in giving the notion of the rate of correlation decay a proper meaning
(see, for example, \cite{CoEc_JSP04}). With the standard definition of a 
correlation
function for two observables $\varphi$ and $\psi$,
\begin{equation}\label{CorrelFunc}
C_{\varphi,\psi}(n) = \int_I{\varphi(f^n(x)) \psi(x) d\mu} - \int_I{\varphi(x) d\mu} 
\int_I{\psi(x) d\mu},
\end{equation}
the exponential decay rate $\alpha_{\varphi,\psi}$ 
governing the asymptotic behaviour of the correlation function may be
formally introduced by
\begin{equation*}\label{decaycorr}
\alpha_{\varphi,\psi}= \sup\{s:\limsup_{n\rightarrow \infty} |\exp(sn) C_{\varphi,\psi}(n)|<\infty\} \, .
\end{equation*}
The rate $\alpha_{\varphi,\psi}$ is sensitive to the choice of $\varphi$ and 
$\psi$ and can be made arbitrarily small or large by special choices for the observables 
\cite{CraCa_PHYD83}. However, it is possible to define a rate of decay with respect  to ``typical'' observables from  some linear space~$\mathcal{H}$ 
\begin{equation*}\label{decayrate}
\alpha_{\mathcal{H}} = \inf\{\alpha_{\varphi,\psi} : \varphi,\psi \in {\mathcal{H}}\} \, . 
\end{equation*}
We will call this quantity the 
\oemph{mixing rate}. Note that this rate still depends,
for instance, on the degree of smoothness shared by the observables in ${\mathcal H}$. 

One approach for determining bounds on the  mixing rate $\alpha_{\mathcal{H}}$ relies on reformulating 
\eref{CorrelFunc} in terms of the Perron-Frobenius operator. 
The rate $\alpha_{\mathcal{H}}$ 
is then determined by the subleading eigenvalue or by the essential spectral radius
of the Perron-Frobenius operator. The desired relation between the Lyapunov exponent and the 
mixing rate expressed in \Eref{ac} then becomes a lower bound for the subleading eigenvalue.
There is a considerable body of literature on upper bounds 
for spectral values 
(see, for example, \cite{Live_AM95,Kell_IJBC99,BrLuSt_ASENS03}) providing useful tools to establish ergodic properties
of dynamical systems. 
However, to the best of our knowledge hardly any nontrivial lower bounds exist 
(see, however, \cite{Naud_AnnHP09},  where an exponential lower bound for correlation functions of suspension semiflows is given). 

The main tool to establish the desired inequality consists in studying the properties
of the generalised Perron-Frobenius operator $\mathcal{L}_{\beta}$ with potential $-\beta \ln|f^{\prime}|,$
\begin{equation}\label{eq:RPF}
 (\mathcal{L}_{\beta} h)(x) = \sum_{y \in f^{-1}(x)} \frac{h(y)}{|f^{\prime}(y)|^{\beta}} \, . 
\end{equation}
For $\beta=1$ this expression reduces to the Perron-Frobenius operator which has leading 
eigenvalue one.
The subleading eigenvalue $\lambda_1$  (or the essential spectral radius, if no subleading 
eigenvalue exists) determines the mixing rate $\alpha_{\mathcal{H}}=-\ln|\lambda_1|$.
In addition, the derivative of the largest eigenvalue $\nu_0(\beta)$ of \Eref{eq:RPF}
with respect to $\beta$, that is, the derivative of the topological pressure,
determines the Lyapunov exponent 
(see, for example, \cite{BoRa_PD87} for a basic exposition). 
Finally, the required estimate follows from the convexity
of the topological pressure and gives a lower bound for the subleading eigenvalue of
the Perron-Frobenius operator. 

The main challenge is, of course, to put these ideas into
practice. Restricting to piecewise linear Markov maps considerably reduces 
the need to worry about mathematical subtleties, as the operator $\mathcal{L}_\beta$ admits finite-dimensional
matrix representations when considering observables consisting of piecewise analytic functions. 
Thus, at a computational level all technical details reduce to
straightforward matrix manipulations (see \Sref{sec:2}), allowing us to keep the presentation 
elementary and, at the same time, making the underlying ideas transparent.
To keep the presentation self-contained  basic properties of piecewise linear 
Markov maps  together with properties of $\mathcal{L}_\beta$ are summarised in the Appendix. 
In \Sref{sec:3} we are addressing the question whether the bounds presented in \Sref{sec:2}
hold for general (non-linear) expanding Markov maps. We have compelling evidence
that the estimate breaks down if analytic observables are considered 
(see also \cite{ChPaRu_PRL90}), but 
the validity of our proposition can be restored if observables of bounded variation are
considered.

\section{Piecewise linear Markov maps}\label{sec:2}

Let us consider a topologically mixing piecewise linear expanding Markov map $f:I\rightarrow I$ 
with respect to the partition $\mathcal{P}=\{I_1,\ldots I_N\}$ of the interval $I$ 
(see Definition~\ref{defn:MarkovMap} for a formal account). 
Denote by $f_k:=f_{I_{k}}$ the $k$-th branch of $f$.
A nice property of 
such maps is that the corresponding transfer operators have finite matrix
representations. In fact, the space of piecewise polynomial functions of degree less than $M$, that is, the space of functions of the form 
$x\mapsto \sum_{k=1}^N \sum_{m=0}^M a_{k m} x^m \chi_{k}(x)$, where $\chi_k$ is the characteristic 
function of the interval $I_k$, is an $\mathcal{L}_{\beta}$-invariant subspace. 
While it appears to be difficult to trace the earliest reference
for this result, 
to the best of our knowledge the first  
experimentally relevant application appeared in the context of
power spectra for intermittent dynamics (see \cite{MoSoOs_PTP81,SoYoOkMo_JSP84}). 

Now, in the natural basis 
of piecewise monomials, the 
operator $\mathcal{L}_{\beta}$ restricted to the invariant subspace mentioned above 
is represented
by the $(M+1)^N\times (M+1)^N $  block upper triangular matrix 
 
\begin{equation}\label{blockmatrix}
\left(
\begin{array}{cccc}
T^{(00)}(\beta)
 & \quad T^{(01)}(\beta) & \quad \ldots &
T^{(0M)}(\beta) \\
0 & \quad T^{(11)}(\beta) & \quad  \ddots & \vdots \\
\vdots &  \ddots & \quad \ddots & \quad T^{(M-1 M)}(\beta) \\
0 & \cdots & 0 & \quad T^{(M M)}(\beta) 
\end{array}
\right) \, .
\end{equation}

A calculation similar to the one used to obtain the matrix representation $T^{(0 0)}(1)$ of $\mathcal{L}_1$ on the space of piecewise constant functions 
(see, for example, \cite[p.~176]{BoyaGora:97}),   
shows that the matrix elements of the block matrices $T^{(mn)}(\beta)$ are given in terms of the 
slopes $\gamma_k$, the intercepts $d_k$ of the branches $f_k$ of the map $f$,  
 and the topological transition matrix $(A_{kl})_{1 \leq k,l\leq N}$ induced by $f$ and $\mathcal{P}$
(see \Eref{appeqa}) as follows:
\begin{equation}\label{melements}
T_{kl}^{(mn)}(\beta)= 
\frac{A_{lk}}{|\gamma_l|^\beta \gamma_l^n}\cdot 
 (-d_l)^{n-m}
{n \choose n-m} \, .
\end{equation}
The eigenvalues are determined by the diagonal blocks 
$T^{(mm)}(\beta)$ with matrix elements given by
the first factor in \Eref{melements}. 

Note that since the underlying map is assumed to be topologically mixing, the matrices $T^{(mm)}(\beta)$ are 
irreducible and aperiodic for even $m$. The Perron-Frobenius Theorem 
(see, for example, \cite[p.53]{Gantmacher:2000} or \cite[p.536]{Lancaster:69}),
now guarantees that $T^{(mm)}(\beta)$ has a simple, positive eigenvalue, larger (in modulus) than all other eigenvalues, which we denote by $\nu_m(\beta)$.  
Now, $\nu_0(\beta)$ determines the topological pressure, given by
$P(\beta)=\ln\nu_0(\beta)$, which has the following well-known properties.

\begin{lem}\label{prop:TopPres}\noindent
\begin{itemize}
\item[(i)] $P(1) = 0$;
\item[(ii)]  $P(\beta)$  is a convex function of $\beta$; 
\item[(iii)] $P(\beta)$ is analytic in $\beta$; 
\item[(iv)] $P'(1)=-\Lambda$.
\end{itemize}
\end{lem}

These properties have been established for a rather large class 
of dynamical systems (see, for example, \cite{Kell:98}). In our context, however, they follow easily 
from the matrix representation above using elementary methods. As an example, we shall derive property (ii).
To simplify notation, we shall denote the first block matrix $T^{(00)}(\beta)$ by $T(\beta)$ 
for the remainder of the argument.
Observing that $P(\beta)=\ln \nu_0(\beta)$ is obtained from
$\ln \mbox{Tr} \left((T(\beta))^n\right)/n$ as
$n$ tends to infinity, convexity of the topological pressure, that is, 
\[
P(\beta_1 (1-t)+ \beta_2 t) \leq
(1-t) P(\beta_1) + t P(\beta_2) \qquad (0\leq t\leq 1),
\] 
follows from a simple estimate using the H\"older inequality 
\begin{eqnarray*}
\fl & &\frac{1}{n} \ln  \mbox{Tr} \left(
(T(\beta_1(1-t)+\beta_2 t))^n \right)\nonumber\\
\fl &=& \frac{1}{n} \ln \sum_{l_0,l_1,\ldots, l_{n-1}}
T_{l_0 l _1}(\beta_1 (1-t) + \beta_2 t)
T_{l_1 l _2}(\beta_1 (1-t) + \beta_2 t)
\cdots
T_{l_{n-1} l _0} (\beta_1 (1-t) + \beta_2 t) 
\nonumber \\
\fl &=&  \frac{1}{n} \ln \sum_{l_0,l_1,\ldots, l_{n-1}}
\left(T_{l_0 l _1}(\beta_1)\right)^{1-t} 
\left(T_{l_0 l _1}(\beta_2)\right)^{t}
\cdots
\left( T_{l_{n-1} l _0}(\beta_1)\right)^{1-t}
\left( T_{l_{n-1} l _0}(\beta_2)\right)^{t} \nonumber \\
\fl 
&\leq& \frac{1}{n} \ln \left( \sum_{l_0,l_1,\ldots, l_{n-1}}
T_{l_0 l _1}(\beta_1) 
\cdots 
T_{l_{n-1} l _0}(\beta_1) \right)^{1-t}
\left(\sum_{l_0,l_1,\ldots l_{n-1}}
T_{l_0 l _1}(\beta_2) 
\cdots 
T_{l_{n-1} l _0}(\beta_2) \right)^t \nonumber \\
\fl
&=& (1-t) \frac{1}{n} \ln \mbox{Tr} \left((T(\beta_1))^n\right)
+ t \frac{1}{n}\ln \mbox{Tr} \left( (T(\beta_2))^n \right)\, .
\end{eqnarray*}
Here we have used that \Eref{melements} implies
$T_{kl}(\beta_1 +\beta_2)=
T_{kl}(\beta_1)T_{kl}(\beta_2)$ as well as
$T_{kl}(t \beta)=(T_{kl}(\beta))^t$. 
The other statements of the
lemma can be derived in a similar way, or can be found in standard textbooks 
on ergodic theory.

To establish a relation between the correlation decay, that is, 
between the eigenvalues of the operator \eref{eq:RPF} for $\beta=1$, and the 
Lyapunov exponent, note that the largest eigenvalue of the Perron-Frobenius operator
is given by $\nu_0(1)=1$, while 
$\nu_2(1)$ is a positive eigenvalue, which provides a lower bound for the 
subleading eigenvalue of $\mathcal{L}_{\beta}$.
Thus, on the one hand
\begin{equation}\label{ba}
\alpha_\mathcal{H} \leq -\ln \nu_2(1) \, .
\end{equation}
On the other hand $T^{(22)}(\beta)=T^{(00)}(\beta+2)$ by \Eref{melements}, which implies 
\begin{equation}\label{bb}
\nu_2(\beta)=\nu_0(\beta+2) \, .
\end{equation}
Hence,  using the properties of the topological pressure
in Lemma~\ref{prop:TopPres}, the relations \eref{ba} and \eref{bb}
yield 
\begin{equation}\label{eq:aH_lambda3}
\alpha_\mathcal{H} \leq -P(3)\leq (3-1) \Lambda \, .
\end{equation}
See \Fref{fig:TopPressure} for a graphical illustration of this result. 
\begin{figure}[h!]
  \centering
    \includegraphics[angle=0,width=.70\textwidth]{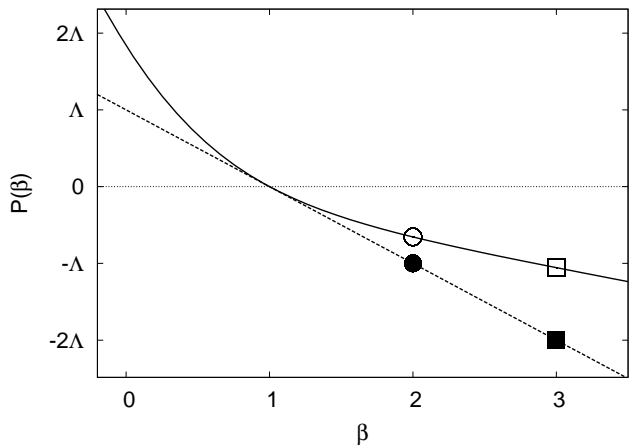}
  \caption{Schematic representation of the topological pressure $P(\beta)$ and 
graphical illustration of the estimates (\ref{eq:aH_lambda3}) and (\ref{eq:aH_lambda2}).
\label{fig:TopPressure}}
\end{figure}

Note that if all slopes $\gamma_k$ of $f$ have the same sign, then we have 
$T^{(11)}(\beta)
=\mbox{sign}(\gamma_k)T^{(00)}(\beta+1)$. Thus, 
we can apply the previous arguments to $|\nu_1(1)|$
to obtain the following improved estimate
\begin{equation}\label{eq:aH_lambda2}
\alpha_\mathcal{H} \leq  -\ln |\nu_1(1)| =  -\ln \nu_0(2) =  
-P(2)\leq (2-1) \Lambda \, .
\end{equation}
To summarise, we have shown the following. 

\begin{prop}\label{thm:main}
Let $f: I \rightarrow I$ denote a topologically mixing 
piecewise linear expanding Markov map.
If $\mathcal{H} = \mathcal{H}(D_R)$ denotes a space of piecewise analytic observables 
(see Appendix, Definition~\ref{def:Hspace}) then
the mixing rate is bounded in terms of the Lyapunov exponent $\Lambda$ with respect to the piecewise constant invariant density, by
\begin{equation}\label{ineq}
\alpha_{\mathcal{H}} \leq 2 \Lambda \, .
\end{equation}
If all slopes have the same sign the sharper estimate
\begin{equation}\label{sharp}
\alpha_{\mathcal{H}} \leq  \Lambda
\end{equation}
holds.
\end{prop}

\begin{rem}
The assumption that $f$ be topologically mixing is sufficient but not
necessary. Indeed, there exist piecewise linear expanding Markov
maps $f$ with the following properties: the map $f$ is not topologically
mixing, yet exhibits exponential decay of correlations and
the conclusions of Proposition \ref{thm:main} hold.
\end{rem}

The simplest examples for which the bounds are achieved are
the tent map ($\alpha_{\mathcal{H}}=2\Lambda$) and the doubling map
($\alpha_{\mathcal{H}}=\Lambda$). 
It is finally worth mentioning that the condition $P(1)=0$, that is, the absolute continuity
of the reference measure with respect to Lebesgue measure, is not essential
for Proposition~\ref{thm:main} to be valid. The conclusions also hold, for instance, for certain Gibbs measures with respect to piecewise 
constant potentials.

\section{Remarks on general expanding Markov maps}\label{sec:3}

The setup of piecewise linear Markov maps is rather special. One may thus be tempted 
to ask whether a result like Proposition~\ref{thm:main} extends, say, to Markov maps
with finite curvature. The previous considerations are based 
on \Eref{bb}, a result which has been exploited previously in a slightly more
restricted setup \cite{BaHeMePo_PRA88} to conjecture an exact relation between 
correlation decay and generalised Lyapunov exponents. It has already been
demonstrated that such an identity breaks down for maps with finite curvature
\cite{ChPaRu_PRL90}. This fact, taken on its own, however, does not prevent the validity of a generalisation of
Proposition~\ref{thm:main}, making a study of how curvature
affects the previous considerations a worthy task.

Using an Ulam-like construction any expanding Markov map can be approximated 
by a piecewise linear map (see, for example, \cite{BaIsSc_AIHP95}). 
Let $F:I\rightarrow I$ denote a piecewise expanding, but  
not necessarily linear Markov map with Markov
partition $\{I_1,\ldots,I_N\}$. Using the cylinder sets $U_{i_0,\ldots, i_{n-1}}
= \cap_{k=0}^{n-1} F^{-k}(I_{i_k})$ we may introduce a piecewise linear approximation
$f_n: I\rightarrow I$ by the following construction. The map $f_n$ linearly interpolates
$F$ on each cylinder set, that is, $f_n(U_{i_0,\ldots, i_{n-1}}) = U_{i_1,\ldots, i_{n-1}}$
(see \Fref{fig:Counterexample}(a)).  It is a straightforward exercise 
to show that the Lyapunov exponent of $f_n$ tends to the Lyapunov exponent of $F$ (with respect to the absolutely continuous invariant measure) as $n$ tends to infinity. 
However, the analysis of the mixing rate requires greater care.
While Proposition~\ref{thm:main} is still valid for any order $n$ of the approximation, it is far from obvious whether the proposition
is valid for a general expanding map.

To illustrate this point we consider a simple example,
a family of full branch piecewise M\"obius maps $F_c$ defined  on $[-1,1],$ 
\begin{equation}\label{oscarmap}
F_c(x)=\frac{1-2(c+1)|x|}{1+2c|x|} \, .
\end{equation}
We restrict the parameter to $c \in (-1/4,1/2)$, in order to guarantee expansivity.
\Fref{fig:Counterexample}(a) depicts the map $F_c$ for $c=-0.22$.

The leading part of the spectrum of the corresponding Perron-Frobenius
operator considered
on the space of analytic observables
can be approximated using a spectral approximation method. The basic idea of this method is to approximate 
$\mathcal{L}_{\beta}$ by an $n\times n$ square matrix $\Pi_n\mathcal{L}_{\beta}\Pi_n$, 
where $\Pi_n$ denotes the projector that sends a function to its Lagrange-Chebyshev interpolating polynomial of degree $n-1$.
This method is easily implemented and, moreover, it is possible to show (see \cite{Oscar}) that 
the eigenvalues of $\Pi_n\mathcal{L}_{\beta}\Pi_n$ converge exponentially fast to the 
eigenvalues of $\mathcal{L}_{\beta}$. 
Using this method the leading eigenvalues 
of the Perron-Frobenius operator and their dependence on $c$ are
easily obtained (see \Fref{fig:Counterexample}(b)). 
A minimum for the subleading eigenvalue occurs at about
$c=-0.11$.
The corresponding numerical value
reads $\lambda_1 \approx 0.10415 $ resulting in a mixing rate
$\alpha_{\mathcal{H}}=-\ln |\lambda_1| \approx 2.2619$. The corresponding Lyapunov 
exponent (which hardly depends on the parameter $c$) is computed using
the numerical approximation
of the invariant density. The numerical value is 
$\Lambda \approx 0.685$ so that the inequality \eref{ineq} is clearly violated.

\begin{figure}[h!]

$
\hspace{-.08\textwidth}
\begin{array}{cc}
\includegraphics[width=.65\textwidth]{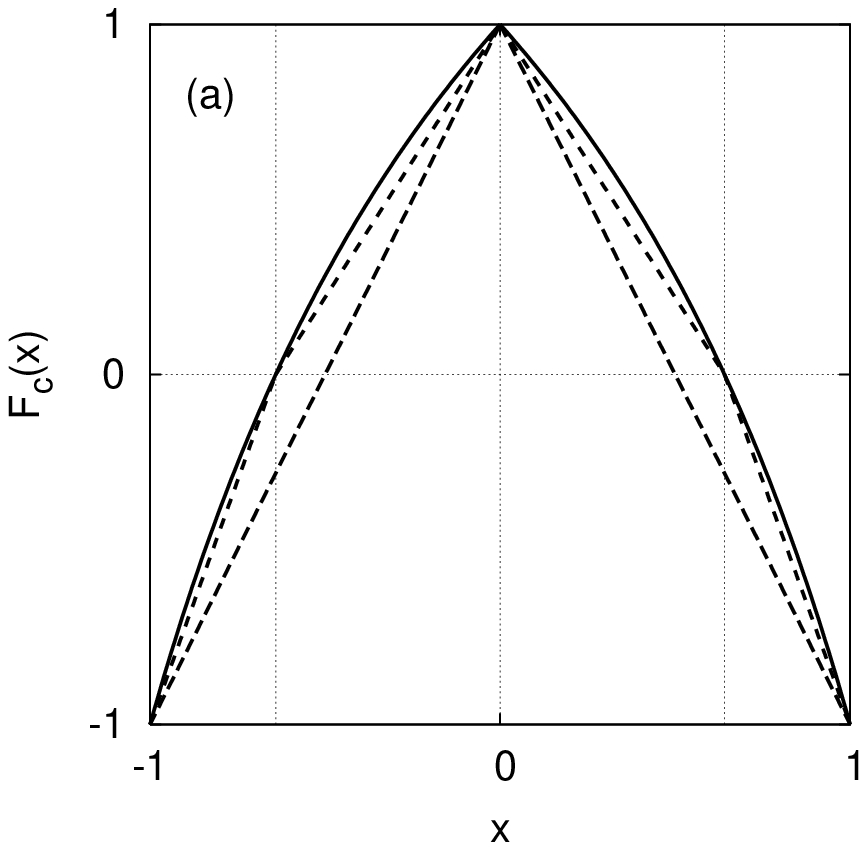} &
\hspace{-.12\paperwidth}
\includegraphics[width=.65\textwidth]{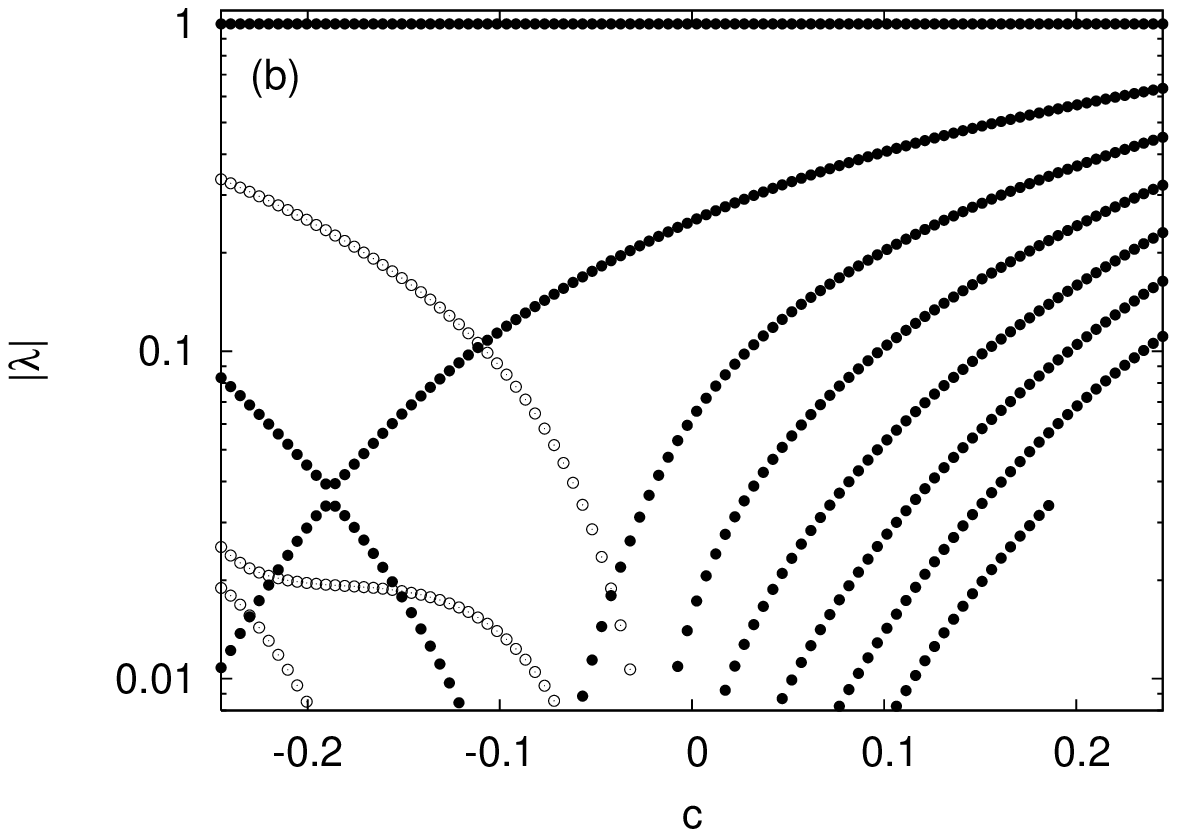}
\end{array}$
\caption{Left: (a) Diagrammatic view of the  M\"obius map \eref{oscarmap} for $c<0$ (solid).
Dotted lines indicate the cylinder sets of the first generations. Broken lines
show the first two piecewise linear approximations $f_{1}$ and $f_{2}$. \newline
Right: (b) The largest eigenvalues (in modulus) of the Perron-Frobenius
operator for the map $F_c$ in dependence on $c$, as obtained using the 
Lagrange-Chebyshev method with truncation order $n=25$. Positive/negative
eigenvalues are indicated by filled/open symbols. \label{fig:Counterexample}}
\end{figure}

In order to understand why the reasoning at the beginning of this section fails, that is, 
why an approximation of the map $F_c$ by a piecewise linear Markov map fails to produce the correct mixing rate, let us consider increasingly finer
piecewise linear approximations $f_n$ of the map \eref{oscarmap}, 
see \Fref{fig:Counterexample}(a).
For each map $f_n$ the Perron-Frobenius operator restricted 
to piecewise polynomial functions has a finite matrix representation (see \eref{blockmatrix}). 
For the remainder of the section we shall only deal with the case 
$\beta = 1$ and refer to $T(\beta)$ (and $\nu_m(\beta)$) as $T$ (and $\nu_m$).
\Fref{fig:EVs_Efunctions}(a) shows the numerical results
for the leading eigenvalues of the diagonal block $T^{(11)}$
and  $T^{(22)}$, that is, $\nu_{1}$ and
$\nu_{2}$, respectively, for increasing level of approximation $n=1,\ldots,6$. For comparison we display the subleading eigenvalue 
$\nu^{(subl)}_0$ of $T^{(00)}$ as well.
For every level $n$ of approximation, $\nu_{2}$ gives the subleading 
eigenvalue $\lambda_1$ of the Perron-Frobenius operator of the piecewise
linear approximation, and these values seem to converge as $n$ tends to infinity.
The values and the limit are  larger than $\exp(-2\Lambda)$, meaning that the 
inequality \eref{ineq} in the Proposition \ref{thm:main} is satisfied, as expected. 
The eigenfunction $u_n:[-1,1]\to \mathbb{R}$ of the corresponding eigenvalue 
$\nu_{2}$ is a quadratic polynomial on each element of the partition, but
it develops an increasing number of discontinuities between different intervals of 
the increasingly finer partition. 
These eigenfunctions do not seem to converge to a smooth limit 
(see \Fref{fig:EVs_Efunctions}(b)). In fact, unlike the invariant density there is no
reason why the limit should be smooth. The numerical experiment suggests that we end
up with a function of bounded variation.

\begin{figure}[h!]

\hspace{-0.6\textwidth}
\begin{center}
\includegraphics[width=1\textwidth]{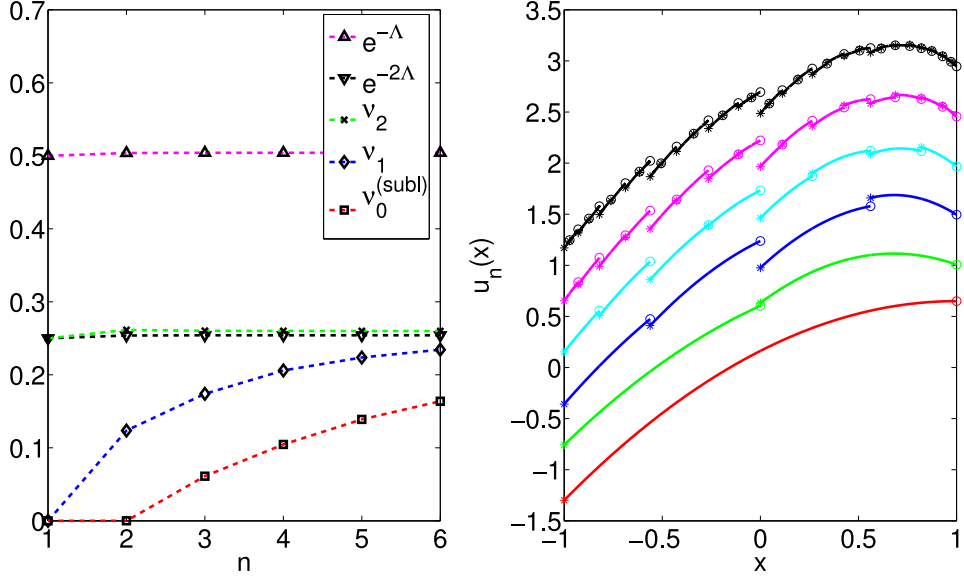}
\end{center}
 \caption{Left: (a) Leading eigenvalues 
$\nu_1$ and $\nu_2$ of the  corresponding matrix blocks
$T^{(11)}$ and $T^{(22)}$
for a piecewise linear approximation of the map \eref{oscarmap} with $c=-0.11$, as a function
of the level of approximation. The subleading eigenvalue of  
$T^{(00)}$, $\nu_0^{(subl)}$, is displayed as well.
For comparison, $\exp(-\Lambda)$ and $\exp(-2\Lambda)$ are depicted as well.
The broken lines are a guide for the eye. \newline
Right:  (b) Eigenfunction $u_n$ corresponding to $\nu_2$ for 
$n=1,\ldots,6$ with normalisation $\int_I|u_n(x)|\, dx=1$. For clarity, successive
approximations are shifted by $0.5$. The open symbols indicate the discontinuity set of the
eigenfunction, i.e., the increasingly finer Markov partition of the piecewise linear approximaton. 
\label{fig:EVs_Efunctions}}
\end{figure}

It is indeed possible to show and perhaps well known that an
estimate like \Eref{sharp} holds for discontinuous observables.
For that purpose let us consider the Perron-Frobenius operator 
$\mathcal{L}_1$ on the space  of functions of bounded 
variation.
Recall that a function $f:[-1,1] \to \mathbb{R}$ is of bounded variation if it has finite total variation 
$\text{var}(f) = \sup \big\{ \sum_{i=1}^{p} |f(x_i)-f(x_{i-1})|:-1\leq x_0\leq \ldots \leq x_p \leq 1 \big \} < \infty$.
In this setup, the spectrum of the Perron-Frobenius operator associated with expanding maps has been studied in detail (see \cite{Kell_CMP84}).  
In particular, there is an explicit formula for the essential spectral radius given by 
\begin{equation*}
\sigma_{ess} = \lim_{k\to \infty} (\inf \{|(F_c^k)'(x)| : x \in[-1,1] \})^{-1/k} \, .
\end{equation*}
Thus, we have an upper bound for the mixing rate
\begin{equation*}
\alpha_{BV}\leq -\ln \sigma_{ess}=  \lim_{k\to \infty}
\frac{1}{k} \ln \inf \{|(F_c^k)'(x)| : x \in[-1,1] \} \,,
\end{equation*}
which yields the following estimate for the Lyapunov exponent:
\begin{eqnarray*}
\Lambda &=& \frac{1}{k} \int_I \ln |(F_c^k)'(x)| \,d\mu\nonumber \\
& \geq& \frac{1}{k}
\inf \{\ln|(F_c^k)'(x)| : x \in[-1,1] \} \int_I d \mu\nonumber \\
&=& \frac{1}{k} \ln \inf \{|(F_c^k)'(x)| : x \in[-1,1] \} \,. 
\end{eqnarray*}
Thus, for observables of bounded variation we have the following result. 

\begin{prop}\label{bvcor}
Let $f: I \rightarrow I$ be a piecewise monotonic smooth expanding interval map which is mixing with respect to its unique absolutely continuous 
invariant measure.  Then the rate of decay of correlations for
functions of bounded variation is bounded by the Lyapunov exponent
\begin{equation}\label{bvest}
\alpha_{BV}\leq \Lambda \, .
\end{equation}
\end{prop}

In fact, almost identical statements can be found in 
\cite{CoEc_JSP04}, for example, Corollary~9.2. 

\section{Conclusion}\label{sec:4}

There is no simple, straightforward answer to the question about the relation between
Lyapunov exponents and mixing rates. On formal grounds one may argue 
that both quantities probe entirely different and independent aspects of 
a dynamical system, and that no particular relation should be expected.
Lyapunov exponents are determined by properties related to the largest eigenvalue
of the Perron-Frobenius operator. By contrast, correlation decay
depends crucially on properties of the observables, with mixing
rates being related to the subleading part of the spectrum. 
Thus, abstract operator theory on its own does not seem to provide further insight into the  
relation between both quantities. Witness, for example, the doubling map viewed as an analytic map 
on the unit circle.
While its Lyapunov exponent is finite, the Perron-Frobenius operator
has no nontrivial eigenvalue when considered on the space of analytic
functions, that is, correlations of analytic observables decay faster than any exponential
(see, e.g., \cite{Bala:00}).

The argument outlined above, however, is a bit too simplistic. In fact, 
our results on
piecewise linear expanding Markov maps observed via piecewise analytic functions or  general piecewise smooth expanding maps observed via functions of bounded
variation suggest that bounds on the mixing rate in terms of
Lyapunov exponents can be derived provided that specific properties of
the underlying dynamical system are taken into account. 
Estimates of this type rely on nontrivial lower bounds for spectra. As such, they are complementary to
estimates which are available for proving the existence of spectral gaps,
and will thus require completely different approaches.

It turns out that the spirit of the result contained in Proposition \ref{bvcor} 
can be understood by considering 
an observable with a single discontinuity. In order to substantiate this claim 
we have performed numerical simulations on the map \eref{oscarmap} for $c=-0.11$.
We have computed the autocorrelation function \eref{CorrelFunc} for the
observable $\varphi=\psi=\phi_h$ with $\phi_h(x)=x$ if $|x|<1/2$ and
$\phi_h(x)=x-\mbox{sign}(x) h$ if $|x|>1/2$, having a discontinuity of stepsize $h$
at $|x|=1/2$ (see \Fref{fig:4}). Choosing $h=0$, the corresponding observable $\phi_h$ is analytic and the
correlation decay is seen to follow the subleading eigenvalue $\lambda_1$ of the 
Perron-Frobenius operator defined on the space of analytic functions 
(see \Fref{fig:Counterexample}(b)). In the discontinuous case corresponding to $h\neq 0$
the short time initial decay of the correlations still follows the pattern of the 
analytic observable, but
the correlation function now develops an exponential tail which obeys \Eref{bvest}.
The tail becomes more pronounced if the stepsize increases. 
In fact, the mixing rate seems to be very close to the Lyapunov exponent. 
Revisiting the considerations leading to Proposition \ref{bvcor},
it is tempting to surmise that this coincidence is a consequence of large deviation properties of finite time Lyapunov exponents,
since the expression for the essential spectral radius involves
an extreme value of a finite time Lyapunov exponent. 
Thus, the relation between Lyapunov exponents and mixing rates for
observables of bounded variation could be viewed to arise from the same 
mechanism already exploited in \cite{BaHeMePo_PRA88} for analysing the simple case mentioned 
in the introduction.

\begin{figure}[h!]
\begin{center}
\includegraphics[width=.75\textwidth]{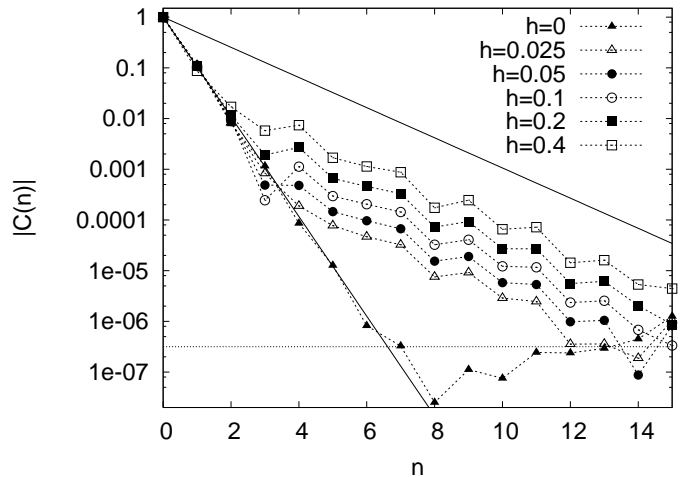}
\end{center}
\caption{Normalised autocorrelation function of the discontinuous variable 
$\phi_h$ with different
stepsizes $h$ for the map $F_c$ \eref{oscarmap} with $c=-0.11$. The straight lines represent an
exponential decay with rate $\ln|\lambda_1|$ and $-\Lambda$. Ergodic averages have been
computed as time averages of a series of length $2\times 10^4$ for $5\times 10^8$ uniformly
distributed initial conditions, skipping a transient of length $100$. The horizontal dotted
line indicates the order of magnitude of statistical errors induced by the finite
ensemble size.\label{fig:4}}
\end{figure}

This simple demonstration gives support to the folklore that correlation 
decay is linked with Lyapunov exponents.
Even if a real world phenomenon is sufficiently well-modelled by a smooth dynamical
system for which to date no link between correlation decay and
Lyapunov exponents can be established, one should keep in mind that modern data processing 
inevitably involves digital devices, which correspond to discontinuous observations.
Therefore, in formal terms observables of bounded variation could be the
relevant class for applications and in these cases Proposition \ref{bvcor} applies.
 
At an intuitive level it is easy to understand  
why discontinuous observations
result in correlation decay related to Lyapunov exponents. A discontinuous
observable is able to distinguish between different ``microstates'' at a ``macroscopic'' level, that is, a discontinuous observation is able to
distinguish two states at infinitesimal distance. As the 
distance between two nearby phase space points separated by a discontinuity
grows according to the Lyapunov exponent, 
the sensitivity at the
microscale may be transported to the macroscale, that is, it may filter through to the
correlation function by a discontinuous observation.

In our context the mathematical challenge is 
to establish a relation between mixing rates and Lyapunov exponents for
natural classes of observables, for example, full branch analytic interval maps 
observed via analytic functions. Besides the need for developing tools to
obtain lower bounds for spectra, establishing the relation alluded to above also requires 
a deeper understanding of which dynamical feature causes the point spectrum of
the Perron-Frobenius operator, that is, which signature of the microscopic dynamics 
survives if viewed via analytic observables. This is reminiscent of coarse-graining approaches in statistical mechanics,
for example, the introduction of collective coordinates and quasi-particles. Thus,
tackling the mathematical problem above may well shed some light 
on some of the most fundamental problems in contemporary nonequilibrium statistical physics
of complex systems.

\section*{Acknowledgements}
W.J. gratefully acknowledges support by EPSRC (grant no.\ EP/H04812X/1) and DFG (through SFB910),
as well as the kind hospitality by Eckehard Sch{\"o}ll and his group during the stay 
at TU-Berlin.

\appendix
\setcounter{section}{1}
\renewcommand\thesection{\Alph{section}}

\section*{Appendix: Spectral properties of Markov maps}\label{sec:Appendix}
In this section we shall provide a short account of the technical 
details
to make the results of Section~\ref{sec:2} rigorous. 
The main thrust 
of the argument is to define a suitable 
function space on which the generalised Perron-Frobenius operator
is compact. Results of this type for general analytic Markov maps are not new (see, for example, \cite{ruelle76} or
\cite{mayer84}). The special case of piecewise linear
Markov maps discussed below, where a complete determination of the
spectrum is possible, is probably well known to specialists in the
field. Unfortunately, we are at loss to provide a reference for the results
in Proposition~\ref{app:prop2} below, so we will outline a proof for the
convenience of the reader.  

To set the scene we define what is meant by a piecewise linear Markov map. 
Before doing so we note that by a \oemph{partition} of a closed 
interval $I$ we mean a finite collection of closed
intervals $\{I_1,\ldots,I_N\}$ with disjoint 
interiors, that is, 
$\mbox{int}(I_k)\cap\mbox{int}(I_l)$ for $k\neq l$, such that 
$\bigcup_{k=1}^NI_k=I$. 

\begin{defn}\label{defn:MarkovMap}
An interval map $f:I\rightarrow I$ is said to 
be a \oemph{Markov map} if there exists 
a finite partition $\{I_k\}_{k=1}^{N}$ of $I$ such that 
for any pair $(k,l)$ either
$f(\mbox{int}(I_k)) \cap \mbox{int}(I_l) = \emptyset$ or 
$\mbox{int}(I_l) \subseteq f(\mbox{int}(I_k))$. If this is the case, 
the corresponding partition will be referred to 
as a \oemph{Markov partition} and  
the $N\times N$ matrix $A$ given by 
\begin{equation}\label{appeqa} A_{k l}=
  \begin{cases}
    1 & \text{if  $\mbox{int}(I_l) \subseteq f(\mbox{int}(I_k))$} \\
    0 & \text{otherwise}
  \end{cases}
\end{equation}
will be called the \oemph{topological transition matrix} of the Markov map $f$. 

A Markov map $f$ with Markov partition $\{I_k\}_{k=1}^{N}$ is said to be 
\oemph{expanding} if $|f'(x)|>1$ for all $x\in \mbox{int}(I_k)$. It is
said to be \oemph{piecewise linear} if $f'$ is constant on each element of the 
Markov partition, that is, $f'(x)=\gamma_k$ for all $x\in \mbox{int}(I_k)$.

Finally, we call an expanding Markov map with topological transition matrix $A$ 
\oemph{topologically mixing}\footnote{This is a slight abuse
of terminology, since its use is usually restricted to continuous maps. However, it serves the same purpose as in the continuous setup 
as it guarantees the existence of a spectral gap for the corresponding transfer operator (see Corollary~\ref{app:cor}).} 
if there is a positive integer $p$ such that each 
entry of the matrix $A^p$ is strictly positive. 
\end{defn}

In what follows we shall concentrate on 
topologically mixing piecewise linear expanding Markov
maps. Our aim is to define suitable spaces of observables on which 
the associated \oemph{generalised Perron-Frobenius operator} 
or \oemph{transfer operator}  
\begin{equation}\label{app:eq:RPFapp}
 (\mathcal{L}_{\beta} h)(x) = \sum_{y \in f^{-1}(x)} \frac{h(y)}{|f^{\prime}(y)|^{\beta}}
\end{equation}
(see \Eref{eq:RPF}) is well defined and has nice spectral properties. 
It turns out that these spaces 
can be chosen from spaces of functions which 
are piecewise analytic. 

In order to define these spaces we require some more notation. 

\begin{defn}\label{def:Hspace}
Let $D$ denote an open disk in the complex plane. 
\begin{itemize}
  \item[(i)] We write $H^\infty(D)=\all{h: D \rightarrow \mathbb{C} }
                {h \mbox{ holomorphic and  $\sup_{z \in D}|h(z)|<\infty$}}$
     to denote the space of bounded holomorphic functions on $D$. 
     This is a Banach space when equipped with the 
      norm $\norm{h}{H^\infty(D)}=\sup_{z \in D}|h(z)|$.
   \item[(ii)] We use $\mathcal{H}(D)=\bigoplus_{k=1}^N H^\infty(D)$ to denote the space of 
            $N$-tuples $(h_1,\ldots,h_N)$ of bounded holomorphic
            functions on $D$. This is a
             Banach space when 
             equipped with the norm $\norm{h}{\mathcal{H}(D)}=\max \{ \|h_k\|_{H^\infty(D)} : k=1,\ldots, N\}$.   
  \end{itemize}
\end{defn}

The desired space of observables will now be defined by linking the 
disk $D$ occurring in the definition above 
to the dynamical system as follows. 
Given a piecewise linear expanding Markov map with 
Markov partition $\{I_k\}_{k=1}^{N}$, let us 
denote by $\varphi_{kl}: I_k \rightarrow I_l$ the inverse branch of the Markov map from partition element $I_k$ into the partition element 
$I_l$ as well as its obvious analytic continuation to the complex plane. Observe now that, since the map is expanding, 
all inverse branches are contractions. We can thus choose two concentric disks 
$D_r$ and $D_R$ of radius $r>0$ and $R>r$, respectively, 
such that 
\begin{equation}
\label{app:eq:adapted}
\varphi_{kl}(D_R)\subset D_r  \quad \text{for all inverse branches $\varphi_{kl}$}\,.
\end{equation}

It turns out that ${\mathcal H}(D_R)$ is a suitable space of observables for the map, in the sense that the associated 
transfer operator (\ref{app:eq:RPFapp}) is a well defined bounded operator on ${\mathcal H}(D_R)$. 
This is the content of the following result. 

\begin{prop}\label{app:prop1}
Given a piecewise linear expanding Markov map $f$ with 
topological transition matrix $A$ and 
inverse branches $\varphi_{kl}$, suppose that the disk $D_R$ is 
chosen as above. 
Then, for any real $\beta$, the transfer operator ${\mathcal L}_\beta$ 
is a well defined bounded operator 
from ${\mathcal H}(D_R)$ into itself and is given by
\begin{equation}\label{app:eq1}
\left({\mathcal L}_{\beta} h\right)_{k}(z) =\sum_{l}A_{lk}|\varphi_{kl}'(z)|^\beta
h_l(\varphi_{kl}(z)) \, .
\end{equation}
\end{prop}
\begin{proof}
The representation (\ref{app:eq1}) follows from a short calculation using the definition (\ref{app:eq:RPFapp}) of ${\mathcal L}_\beta$. 
Since $|\varphi_{kl}'(z)|^\beta=|\gamma_l|^{-\beta}$ is constant and the disk $D_R$ satisfies  
(\ref{app:eq:adapted}), the operator maps ${\mathcal H}(D_R)$ to ${\mathcal H}(D_R)$. In order to see that 
${\mathcal L}_\beta: {\mathcal H}(D_R)\to {\mathcal H}(D_R)$ is bounded observe that if $h\in {\mathcal H}(D_R)$ with 
$\norm{h}{{\mathcal H}(D_R)}\leq 1$, then 
\[ 
\norm{{\mathcal L}_\beta h}{{\mathcal H}(D_R)}
=\max_{k}\norm{ ({\mathcal L}_\beta h)_k}{{H}^\infty(D_R)}
\leq \max_{k} \sum_{l}A_{lk}|\gamma_l|^{-\beta}<\infty\,.\qedhere
\]
\end{proof}

\begin{rem} The space ${\mathcal H}(D_R)$ is not the only suitable space
  of observables. Restricting to one and the same disk of analyticity 
for each branch, however, simplifies notation. More general spaces are discussed in 
\cite{BaJe_AM08} and \cite{BJ08}. 
\end{rem}

Next we shall explain why the spectrum of ${\mathcal L}_\beta$ viewed
as an operator on ${\mathcal H}(D_R)$ is given by the eigenvalues of
the diagonal blocks $T^{(mm)}(\beta)$ for $m\in {\mathbb N}_0$
defined in (\ref{blockmatrix}). 

The key ingredient of the proof of this statement 
is a factorisation of the transfer
operator together with an approximation argument. 

In order to explain the factorisation of the transfer
operator we observe that in 
(\ref{app:eq1}) the argument of $h_l$, that is 
$\varphi_{kl}(z)$, is contained in the 
smaller disk $D_r$ because of (\ref{app:eq:adapted}). 
We can thus use (\ref{app:eq1})
to define the  operator on a larger function space, namely 
${\mathcal H}(D_r)$. Note that the space is
`larger' as analyticity is only required on a smaller disk $D_r \subset
D_R$. We shall write $\tilde{\mathcal L}_\beta : 
{\mathcal H}(D_r) \rightarrow {\mathcal H}(D_R)$ 
in order to distinguish this lifted operator from the one occurring in
Proposition~\ref{app:prop1}. Note that the arguments in this
proposition can be adapted to show that $\tilde{\mathcal L}_\beta$ is a
bounded operator.   

It is tempting to think of ${\mathcal L}_\beta$ and 
$\tilde{\mathcal L}_\beta$ as being essentially the same, since they are
given by the same functional expression. However, as operators the two
are different as the latter is defined on a larger
domain. Yet, both operators are related by restriction. In
order to give a precise formulation of this fact we introduce a
bounded embedding operator $\mathcal J$ which maps the smaller space
${\mathcal H}(D_R)$ injectively into the larger space ${\mathcal H}(D_r)$. 
To be precise ${\mathcal J} : {\mathcal H}(D_R) \rightarrow {\mathcal
  H}(D_r)$ is given by $({\mathcal J} h)_k= Jh_k$, where  
$J:H^\infty(D_R)\to H^\infty(D_r)$ in turn is given by $(Jh)(z)=h(z)$
for $z\in D_r$. 
Note that ${\mathcal J}$ looks superficially like the identity. This,
however, is misleading as argument and image are considered in
different spaces. 

The relation between ${\mathcal L}_\beta$ and $\tilde{\mathcal L}_\beta$ can
now be written as
\begin{equation}
\label{app:Lfactor}
{\mathcal L}_\beta = \tilde{{\mathcal L}}_\beta {\mathcal J}\,. 
\end{equation}
Note that the factorisation above disentangles the intricacies of
the map contained in $\tilde{\mathcal{L}}_\beta$ from its general  
expansiveness contained in $\mathcal{J}$. 

We now turn to the approximation argument. 
For piecewise linear Markov maps the transfer
operator is easily seen to map 
piecewise polynomial functions of degree at most $M$ into  
piecewise polynomial functions of degree at most $M$. This follows from a 
straightforward calculation using the fact that the 
inverse branches are affine functions.

In order to exploit this property of the transfer operator further we
shall introduce a projection operator defined as follows: 
given an analytic function $h$ in $H^\infty(D_R)$ and an integer $M$ 
we use $P_Mh$ to denote the truncated Taylor series expansion 
\[ (P_Mh)(z) =\sum_{k=0}^M \frac{h^{(k)}(z_0)}{k!}(z-z_0)^k\,,\]
where $z_0$ denotes the centre of the disk $D_R$. Clearly, $P_M$ is a
projection operator. 

It turns out that the projections $P_M$ approximate the
embedding $J$ for large $M$ in a strong sense. In order to make this
statement, the heart of the approximation argument alluded to above,
more precise, we observe that, by Cauchy's Integral Theorem, we have for any $h\in
H^\infty(D_R)$ and any $z\in D_r$  
\[ h(z)-(P_Mh)(z)=\frac{1}{2\pi i}\oint_\Gamma\frac{h(\zeta)}{\zeta-z}
\frac{(z-z_0)^{M+1}}{(\zeta-z_0)^{M+1}}\,d\zeta\,, \]
where the contour $\Gamma$ is the positively oriented boundary of a
disk centred at $z_0$ with radius lying strictly between $r$ and $R$.
It follows that the norm of 
$J-JP_M$ viewed as an operator from $H^\infty(D_R)$ to
$H^\infty(D_r)$ satisfies
\[ \norm{J-JP_M}{H^\infty(D_R)\to H^\infty(D_r)}\leq \frac{R}{R-r}\left ( \frac{r}{R}\right )^{M+1}\,.\] 
In particular, we have 
\begin{equation}
\label{app:Japprox}
\lim_{M\to \infty}\norm{J-JP_M}{H^\infty(D_R)\to H^\infty(D_r)}=0\,.
\end{equation}
In order to extend this result to the space of piecewise
analytic functions ${\mathcal H}(D_R)$ we introduce the projection operator 
$\mathcal{P}_M : \mathcal{H}(D_R) \rightarrow \mathcal{H}(D_R)$ 
by setting $\mathcal{P}_M h = (P_M h_1,\ldots, P_M h_N)$. 
The analogue
of (\ref{app:Japprox}) now reads 
\begin{equation}
\label{app:Jcalapprox}
\lim_{M\to \infty}\norm{{\mathcal J}-{\mathcal J}{\mathcal P}_M}{{\mathcal H}(D_R)\to {\mathcal H}(D_r)}=0\,.
\end{equation}
We are now able to combine the factorisation (\ref{app:Lfactor}) with the approximation result above to
prove the main result of this Appendix.  

\begin{prop}\label{app:prop2}
Suppose we are given a piecewise linear expanding Markov map $f$ with 
inverse branches $\varphi_{kl}$ and disks $D_r\subset D_R$ satisfying (\ref{app:eq:adapted}). 
Then, for any real $\beta$, the transfer operator ${\mathcal L}_\beta$ viewed as an operator 
on ${\mathcal H}(D_R)$ is compact and its non-zero eigenvalues (with multiplicities) are precisely the non-zero eigenvalues 
of the matrices $T^{(mm)}(\beta)$ with $m\in {\mathbb N}_0$ given in (\ref{melements}). 
\end{prop}
\begin{proof}
We start by recalling that for every $M\geq 0$ the transfer operator ${\mathcal L}_\beta$ leaves the 
space $\mathcal{P}_M(\mathcal{H}(D_R))$ invariant, that is, ${\mathcal L}_\beta(\mathcal{P}_M(\mathcal{H}(D_R))) 
\subseteq \mathcal{P}_M(\mathcal{H}(D_R))$. Thus 
\[ ({\mathcal I}-\mathcal{P}_M){\mathcal L}_\beta\mathcal{P}_M=0\,,\]
where ${\mathcal I}$ denotes the identity on ${\mathcal H}(D_R)$. 

Using the above equation and the factorisation (\ref{app:Lfactor}) we see that 
\begin{multline*}
 \norm{{\mathcal L}_\beta- \mathcal{P}_M{\mathcal L}_\beta\mathcal{P}_M    }{{\mathcal H}(D_R)}
  =\norm{{\mathcal L}_\beta(\mathcal{I}-\mathcal{P}_M)    }{{\mathcal H}(D_R)}\\
    =\nnorm{\tilde{{\mathcal L}}_\beta\mathcal{J}(\mathcal{I}-\mathcal{P}_M)    }{{\mathcal H}(D_R)}
    \leq \nnorm{\tilde{{\mathcal L}}_\beta}{\mathcal{H}(D_r)\to \mathcal{H}(D_R)}
     \nnorm{\mathcal{J}-\mathcal{J}\mathcal{P}_M    }{{\mathcal H}(D_R)\to \mathcal{H}(D_r)}\,,
\end{multline*}
which, using (\ref{app:Jcalapprox}), implies 
\begin{equation}
\label{app:Lapprox}
\lim_{M\to \infty}\norm{{\mathcal L}_\beta- \mathcal{P}_M{\mathcal L}_\beta\mathcal{P}_M    }{{\mathcal H}(D_R)}=0\,.
\end{equation}
Since $\mathcal{P}_M{\mathcal L}_\beta\mathcal{P}_M$ is a finite-rank operator for every $M$, the limit above implies that
${\mathcal L}_\beta$ is compact. Clearly, the non-zero eigenvalues of  each $\mathcal{P}_M{\mathcal L}_\beta\mathcal{P}_M$ are exactly
the non-zero eigenvalues of the block matrices (\ref{blockmatrix}). 
The remaining assertion, namely that the non-zero spectrum of the transfer 
operator is captured by the non-zero spectra of the finite dimensional matrix representations 
follows from (\ref{app:Lapprox}) together with an abstract spectral approximation result (see \cite[XI.9.5]{DS}). 
\end{proof}

Specialising to topologically mixing Markov maps we obtain the following refinement of the above proposition. 

\begin{cor}\label{app:cor}
Suppose that the hypotheses of the previous proposition hold. If the Markov map $f$ is also topologically mixing, then 
${\mathcal L}_\beta:{\mathcal H}(D_R)\to {\mathcal H}(D_R)$ has a simple positive leading eigenvalue $\nu_0(\beta)$. 
Moreover, this leading eigenvalue is the Perron eigenvalue of the matrix $T^{(00)}(\beta)$. 
\end{cor}
 \begin{proof}
This follows from the previous proposition together with the observation that for $m\geq 1$ the spectral radius $r(T^{(mm)}(\beta))$ of 
the matrix  $T^{(mm)}(\beta)$ is strictly smaller than the Perron eigenvalue of $T^{(00)}(\beta)$. In order to see this note that 
for all $m\geq 1$ we have 
\[ |T^{(mm)}_{kl}(\beta)|\leq C T^{(00)}_{kl}(\beta)\,, \]
where 
\[ C=\frac{1}{\inf_l |\gamma_l|}<1\,.\]
A short calculation shows that for each $k\geq 0$ and each $m\geq 1$ we have 
\[ \norm{ \left (T^{(mm)}(\beta) \right )^k  }{F}\leq C^k \norm{  \left (T^{(00)}(\beta) \right )^k  }{F}\,, \]
where $\norm{\cdot}{F}$ denotes the Frobenius norm. The spectral radius formula now implies that 
\[  r(T^{(mm)}(\beta))\leq C\nu_0(\beta)\,.\qedhere\]
 \end{proof}


\begin{thebibliography}{10}

\bibitem{Abar:96}
H.~D. Abarbanel.
\newblock {\em {Analysis of Observed Chaotic Data (Institute for Nonlinear
  Science)}}.
\newblock Springer, 1997.

\bibitem{BaHeMePo_PRA88}
R.~Badii, K.~Heinzelmann, P.~Meier, and A.~Politi.
\newblock {Correlation functions and generalized Lyapunov exponents}.
\newblock {\em Physical Review A}, 37(4):1323--1328, Feb. 1988.

\bibitem{Bala:00}
V.~Baladi.
\newblock {\em {Positive transfer operators and decay of correlations}}.
\newblock World Scientific Publishing Co. Pte. Ltd, Singapore, 2000.

\bibitem{BaIsSc_AIHP95}
V.~Baladi, S.~Isola, and B.~Schmitt.
\newblock {Transfer operator for piecewise affine approximations of interval
  maps}.
\newblock {\em Annales de l'Institut Henri Poincar\'{e} (A) Physique
  th\'{e}orique}, 62(3):251--265, 1995.

\bibitem{Oscar}
O.~F. Bandtlow.
\newblock {Lagrange-Chebyshev approximation of transfer operators}.
\newblock {\em In preparation}.

\bibitem{BaJe_AM08}
O.~F. Bandtlow and O.~Jenkinson.
\newblock {Explicit eigenvalue estimates for transfer operators acting on
  spaces of holomorphic functions}.
\newblock {\em Advances in Mathematics}, 218(3):902--925, June 2008.

\bibitem{BJ08}
O.~F. Bandtlow and O.~Jenkinson.
\newblock {On the Ruelle eigenvalue sequence}.
\newblock {\em Ergodic Theory and Dynamical Systems}, 28(06):1701--1711, Dec.
  2008.

\bibitem{BoRa_PD87}
T.~Bohr and D.~Rand.
\newblock {The entropy function for characteristic exponents}.
\newblock {\em Physica D: Nonlinear Phenomena}, 25(1-3):387--398, Mar. 1987.

\bibitem{BoyaGora:97}
A.~Boyarsky and P.~Gora.
\newblock {\em {Laws of Chaos: Invariant Measures and Dynamical Systems in One
  Dimension (Probability and its Applications)}}.
\newblock Birkh\"{a}user, 1997.

\bibitem{BrLuSt_ASENS03}
H.~Bruin, S.~Luzzatto, and S.~van Strien.
\newblock {Decay of correlations in one-dimensional dynamics}.
\newblock {\em Annales Scientifiques de l'\'{E}cole Normale Sup\'{e}rieure},
  36(4):621--646, July 2003.

\bibitem{ChPaRu_PRL90}
F.~Christiansen, G.~Paladin, and H.~H. Rugh.
\newblock {Determination of correlation spectra in chaotic systems}.
\newblock {\em Physical Review Letters}, 65(17):2087--2090, Oct. 1990.

\bibitem{CoEc_JSP04}
P.~Collet and J.~P. Eckmann.
\newblock {Liapunov Multipliers and Decay of Correlations in Dynamical
  Systems}.
\newblock {\em Journal of Statistical Physics}, 115(1/2):217--254, Apr. 2004.

\bibitem{CraCa_PHYD83}
J.~D. Crawford and J.~R. Cary.
\newblock {Decay of correlations in a chaotic measure-preserving
  transformation}.
\newblock {\em Physica D: Nonlinear Phenomena}, 6(2):223--232, Jan. 1983.

\bibitem{DS}
N.~Dunford and J.~T. Schwartz.
\newblock {\em {Linear Operators, Part 2: Spectral Theory}}.
\newblock Wiley-Interscience, 1963.

\bibitem{FaIsVu_PLA90}
M.~Falcioni, S.~Isola, and A.~Vulpiani.
\newblock {Correlation functions and relaxation properties in chaotic dynamics
  and statistical mechanics}.
\newblock {\em Physics Letters A}, 144(6-7):341--346, Mar. 1990.

\bibitem{Gantmacher:2000}
F.~R. Gantmacher.
\newblock {\em {Matrix Theory, Vol. 2}}.
\newblock American Mathematical Society, 2000.

\bibitem{Just_PLA90}
W.~Just.
\newblock {Analytical treatment of fluctuation spectra at the symmetry breaking
  chaos transition}.
\newblock {\em Physics Letters A}, 150(8-9):362--368, Nov. 1990.

\bibitem{KaSc:97}
H.~Kantz and T.~Schreiber.
\newblock {\em {Nonlinear Time Series Analysis}}.
\newblock Cambridge University Press, 2004.

\bibitem{Kell_CMP84}
G.~Keller.
\newblock {On the rate of convergence to equilibrium in one-dimensional
  systems}.
\newblock {\em Communications in Mathematical Physics}, 96(2):181--193, 1984.

\bibitem{Kell:98}
G.~Keller.
\newblock {\em {Equilibrium States in Ergodic Theory}}.
\newblock Cambridge University Press, 1998.

\bibitem{Kell_IJBC99}
G.~Keller.
\newblock {Interval maps with strictly contracting Perron-Frobenius operators}.
\newblock {\em International Journal of Bifurcation and Chaos [in Applied
  Sciences and Engineering]}, 09:1777--1783, Sept. 1999.

\bibitem{Kryl:79}
N.~S. Krylov.
\newblock {\em {Works on the Foundations of Statistical Physics (Princeton
  Series in Physics)}}.
\newblock Princeton University Press, 1979.

\bibitem{Kubo_Sci86}
R.~Kubo.
\newblock {Brownian motion and nonequilibrium statistical mechanics.}
\newblock {\em Science (New York, N.Y.)}, 233(4761):330--334, July 1986.

\bibitem{Lancaster:69}
P.~Lancaster.
\newblock {\em {Theory of Matrices}}.
\newblock Academic Press Inc, 1969.

\bibitem{Live_AM95}
C.~Liverani.
\newblock {Decay of Correlations}.
\newblock {\em The Annals of Mathematics Second Series}, 142(2):239--301, 1995.

\bibitem{mayer84}
D.~H. Mayer.
\newblock {Approach to equilibrium for locally expanding maps in $
  \mathbb{R}^k$}.
\newblock {\em Communications in Mathematical Physics}, 95(1):1--15, 1984.

\bibitem{MoSoOs_PTP81}
H.~Mori, B.-C. So, and T.~Ose.
\newblock {Time-Correlation Functions of One-Dimensional Transformations}.
\newblock {\em Progress of Theoretical Physics}, 66(4):1266--1283, Oct. 1981.

\bibitem{Naud_AnnHP09}
F.~Naud.
\newblock {Entropy and Decay of Correlations for Real Analytic Semi-Flows}.
\newblock {\em Annales Henri Poincar\'{e}}, 10(3):429--451, May 2009.

\bibitem{ruelle76}
D.~Ruelle.
\newblock {Zeta-functions for expanding maps and Anosov flows}.
\newblock {\em Inventiones Mathematicae}, 34(3):231--242, Oct. 1976.

\bibitem{SoYoOkMo_JSP84}
B.~C. So, N.~Yoshitake, H.~Okamoto, and H.~Mori.
\newblock {Correlations and spectra of an intermittent chaos near its onset
  point}.
\newblock {\em Journal of Statistical Physics}, 36(3-4):367--400, Aug. 1984.

\bibitem{Spro:96}
J.~C. Sprott.
\newblock {\em {Chaos and Time-Series Analysis}}.
\newblock Oxford University Press, 2001.

\end{thebibliography}
\end{document}